\documentclass[12pt,reqno]{amsart}
\usepackage{amsfonts}

\usepackage{amssymb,amsmath,graphicx,amsfonts,euscript}
\usepackage{color}

\newtheorem{thm}{Theorem}[section]

\newtheorem{rem}[thm]{Remark}

\newtheorem{lemma}[thm]{Lemma}

\allowdisplaybreaks \voffset=-0.2in \numberwithin{equation}{section}

\begin{document}
\bigskip

\centerline{\Large\bf  A note on global regularity results for 2D Boussinesq }
\smallskip

\centerline{\Large\bf  equations with fractional dissipation}

\bigskip

\centerline{Zhuan Ye}

\medskip

\centerline{ School of Mathematical Sciences, Beijing Normal University,}
\medskip

\centerline{Laboratory of Mathematics and Complex Systems, Ministry of Education,}
\medskip

\centerline{Beijing 100875, People's Republic of China}

\medskip

\centerline{E-mail: \texttt{yezhuan815@126.com }}
\medskip

\centerline{\texttt{Tel.: +86 10 58807735; fax: +86 10 58808208.}}

\bigskip
\bigskip

{\bf Abstract:}~~%
In this paper we study the Cauchy problem for the two-dimensional (2D) incompressible Boussinesq equations with fractional Laplacian dissipation and thermal
diffusion. Invoking the energy method and several commutator estimates, we get the global regularity result of the 2D
Boussinesq equations as long as $1-\alpha<\beta<
\min\Big\{\frac{\alpha}{2},\,\,
\frac{3\alpha-2}{2\alpha^{2}-6\alpha+5},
\,\,\frac{2-2\alpha}{4\alpha-3}\Big\}$ with $0.77963\thickapprox\alpha_{0}<\alpha<1$. As a result, this result is a further improvement of the previous two works \cite{MX,YXX}.

{\vskip 1mm
 {\bf AMS Subject Classification 2010:}\quad 35Q35; 35B65; 76D03.

 {\bf Keywords:}
2D Boussinesq equations; Fractional dissipation; Global regularity.}

\vskip .4in
\section{Introduction}\label{intro}
In this paper we study the Cauchy problem for the 2D incompressible Boussinesq equations with fractional Laplacian dissipation in $\mathbb{R}^{2}$
\begin{equation}\label{Bouss}
\left\{\aligned
&\partial_{t}u+(u \cdot \nabla) u+\nu\Lambda^{\alpha}u+\nabla p=\theta e_{2},\\
&\partial_{t}\theta+(u \cdot \nabla) \theta+\kappa\Lambda^{\beta}\theta=0,\\
&\nabla\cdot u=0,\\
&u(x, 0)=u_{0}(x),  \quad \theta(x,0)=\theta_{0}(x),
\endaligned\right.
\end{equation}
where $u(x,\,t)=(u_{1}(x,\,t),\,u_{2}(x,\,t))$ is a vector field denoting
the velocity, $\theta=\theta(x,\,t)$ is a scalar function denoting
the temperature in the content of thermal convection and the density
in the modeling of geophysical fluids, $p$ the scalar pressure and
$e_{2}=(0,\,1)$. Here the numbers $\nu\geq0$, $\kappa\geq0$, $\alpha\geq0$ and $\beta\geq0$ are real parameters. The fractional Laplacian operator $\Lambda^{\alpha}$,
$\Lambda:=(-\Delta)^{\frac{1}{2}}$ denotes the Zygmund operator which is defined through the Fourier transform, namely $$\widehat{\Lambda^{\alpha}
f}(\xi)=|\xi|^{\alpha}\hat{f}(\xi),$$
where
$$\hat{f}(\xi)=\frac{1}{(2\pi)^{2}}\int_{\mathbb{{R}}^{2}}{e^{-ix\cdot\xi}f(x)\,dx}.$$
The fractional Laplacian serves to model many physical phenomena such as overdriven detonations in gases \cite{CLA}. It is also used in some mathematical models in hydrodynamics, molecular biology and finance mathematics,
see for instance \cite{DI}.

Actually, the standard 2D Boussinesq equations (that is $\alpha=\beta=2$)
model geophysical flows such as atmospheric fronts and oceanic circulation, and play an important role in the study of
Raleigh-Bernard convection (see for example \cite{MB,PG} and references therein).
Moreover, there are some geophysical circumstances related to the Boussinesq equations with fractional Laplacian (see \cite{Cap,PG} for details). The Boussinesq equations with fractional Laplacian also closely related equations such as
the surface quasi-geostrophic equation model important geophysical phenomena (see, e.g., \cite{Constantin}). The standard 2D Boussinesq equations and their fractional Laplacian generalizations have attracted considerable attention recently due to
their physical applications and mathematical significance.
Obviously, for case $\mu=\kappa=0$, the system (\ref{Bouss}) reduces to the inviscid  Boussinesq equations, whose global well-posedness of smooth solutions is an outstanding open
problem in fluid dynamics (except if $\theta_{0}$ is a constant, of course) which may be formally compared to the similar problem for the three-dimensional axisymmetric Euler equations with swirl (see \cite{MB}).  In contrast, in the case when $\alpha=\beta=2$, the global well posedness has been shown previously, we refer, for example, to \cite{Can}. Therefore, there are a large number of works devoted to studying the intermediate cases, such as fractional dissipation, partial anisotropic dissipation and so on.
The global regularity to the system
(\ref{Bouss}) for the cases when $\alpha=2$ and $\kappa=0$ or $\beta=2$ and $\mu=0$ were established by Chae \cite{C1} and
by Hou and Li \cite{HL} independently.
By deeply developing the structures of the coupling system, Hmidi, Keraani and Rousset \cite{HK3,HK4} were able to established the global well-posedness result to the system (\ref{Bouss}) with two special critical case, namely
$\alpha=1$ and $\kappa=0$ or $\beta=1$ and $\mu=0$.
The more general
critical case $\alpha+\beta=1$ with $0<\alpha,\,\beta<1$ is extremely difficult. Very recently,  the global regularity of the general
critical case $\alpha+\beta=1$ with $\alpha>\frac{23-\sqrt{145}}{12}\thickapprox 0.9132$ and $0<\beta<1$ was recently examined by Jiu,
Miao, Wu and Zhang \cite{JMWZ}. This result was further improved by Stefanov and Wu \cite{SW}, which requires $\alpha+\beta=1$ with $\alpha>\frac{\sqrt{1777}-23}{24}\thickapprox 0.798$ and $0<\beta<1$. Here we want to state that even in the subcritical ranges, namely $\alpha+\beta>1$ with $0<\alpha,\,\beta<1$, the global regularity of (\ref{Bouss}) is also definitely nontrivial and quite difficult. Actually, to the best of our knowledge
there are only several works concerning the subcritical cases, please refer to \cite{CV,MX,YJW,YX201501,YX201502,YXX}. More precisely, Miao and Xue \cite{MX} obtained
the global regularity for system (\ref{Bouss}) for the case $\nu>0$, $\kappa>0$ and
$$\frac{6-\sqrt{6}}{4}<\alpha<1,\,\,\,1-\alpha<\beta<\min\Big\{\frac{7+2\sqrt{6}}{5}\alpha-2,\,\,
\frac{\alpha(1-\alpha)}{\sqrt{6}-2\alpha},\,\,\,2-2\alpha \Big\}.$$
In addition, Constantin and Vicol \cite{CV} verified the global regularity of the system (\ref{Bouss}) on the case when the thermal diffusion dominates, namely
$$\nu>0,\,\,\,\kappa>0,\,\,\,0<\alpha<2,\,\,\,0<\beta<2,\,\,\,\beta>\frac{2}{2+\alpha}.$$
Recently, Yang, Jiu and Wu \cite{YJW} proved the global regularity of the system (\ref{Bouss}) with
$$\nu>0,\,\,\,\kappa>0,\,\,\,0<\alpha <1,\,\,\,0< \beta<1,\,\,\,\beta>1-\frac{\alpha}{2},\,\,\,
\beta\geq\frac{2+\alpha}{3},\,\,\,\beta>\frac{10-5\alpha}{10-4\alpha}.$$
Here we want to point out that the above two works \cite{CV,YJW} have been improved by the recent two manuscripts \cite{YX201501,YX201502}. In particular, we \cite{YX201502} proved the global well-posedness result for the system (\ref{Bouss}) with
$$\nu>0,\,\,\,\kappa>0,\,\,\,0<\alpha <1,\,\,\,0< \beta<1,\,\,\,\beta>1-\frac{\alpha}{2}.$$
It is also worthwhile to mention that there are numerous studies about the Boussinesq equations with partial anisotropic dissipation, see for example \cite{ACW10,ACW11,ACSWXY,DP3,CW,LLT}.
Many other interesting recent results on the
Boussinesq equations can be found, with no intention to be complete (see,
e.g., \cite{ACWX,CW2,CDJ,D,DP3,HassHm,Hmidi2011,JPL,JMWZ,JWYang,KRTW,LT,LLT,LWZ,WuXu,WXY,XX,YZ} and the references therein).

To complement and improve the existing results described above, this paper continues the previous two works \cite{MX,YXX} to show the global regularity result. Since the concrete values of the constant $\nu,\,\kappa$ play no role in
our discussion, we shall assume $\nu=\kappa=1$ throughout
this paper. Now our main result is the following theorem.
\begin{thm}\label{Th1} Suppose that $0.77963\thickapprox\alpha_{0}<\alpha<1$ and $0<\beta<1$
obeys
\begin{eqnarray}\label{NR}1-\alpha<\beta<
\min\Big\{\frac{\alpha}{2},\,\,
\frac{3\alpha-2}{2\alpha^{2}-6\alpha+5},
\,\,\frac{2-2\alpha}{4\alpha-3}\Big\}.
\end{eqnarray}
Let $(u_{0}, \theta_{0})
\in H^{\sigma}(\mathbb{R}^{2})\times H^{\sigma}(\mathbb{R}^{2})$ for $\sigma>2$,
then the system (\ref{Bouss}) admits a unique global solution such that for any
$T>0$
$$u\in C([0, T]; H^{\sigma}(\mathbb{R}^{2}))\cap L^{2}([0, T]; H^{\sigma+\frac{\alpha}{2}}(\mathbb{R}^{2})),$$
$$\theta\in C([0, T]; H^{\sigma}(\mathbb{R}^{2}))\cap L^{2}([0, T];
H^{\sigma+\frac{\beta}{2}}(\mathbb{R}^{2})).$$
\end{thm}
\begin{rem}\rm
Here we say some words about $\alpha_{0}$ which can be explicitly formulated as
$$\alpha_{0}=\frac{8-(\sqrt[3]{6\sqrt{609}+118}-\sqrt[3]{6\sqrt{609}-118})}{6}\thickapprox
0.77963.$$ By using the well-known Shengjin's Formulas \cite{Fan}, it is easy
to show that $\alpha_{0}$ is a unique real solution to the following
cubic equation
$$2\alpha^{3}-8\alpha^{2}+14\alpha-7=0.$$
\end{rem}
\begin{rem}\rm
The condition $\alpha>\alpha_{0}\thickapprox0.77963$ is weaker than the previous two works \cite{MX,YXX}, where the corresponding conditions are $\alpha>\frac{6-\sqrt{6}}{4}\thickapprox0.887627$ and $\alpha>\frac{21-\sqrt{217}}{8}\thickapprox0.783635$, respectively. Hence, this result
can be regarded as a further improvement of the results in \cite{MX,YXX}.
\end{rem}
\begin{rem}\rm
For technical reasons, the $\beta$ should be smaller than a complicated explicit function. As a matter of fact, it is strongly
believed that the diffusion term is always good term and the larger the
power $\beta$ is, the better effects it produces. Therefore,
we conjecture that the above theorem should hold
for all the cases $\alpha_{0}<\alpha<1$ and $1-\alpha<\beta<1$.
\end{rem}

\vskip .4in
\section{The proof of Theorem \ref{Th1}}\setcounter{equation}{0}
This section is devoted to the proof of Theorem \ref{Th1}. Now let
us to prove our main theorem. First, the local well posedness of the
system (\ref{Bouss}) for smooth initial data is well-known to us
(see for example \cite{MB}), and therefore, it suffices to prove the
global in time {\it a priori} estimate on $[0,\,T]$ for any given
$T>0$. In this paper, all constants will be denoted by $C$ that is a
generic constant depending only on the quantities specified in the
context.

Thanks to the basic energy estimates, we obtain
immediately
\begin{eqnarray}\label{t301}
\sup_{0\leq t\leq T}\|\theta(t)\|_{L^{2}}^{2}+\int_{0}^{T}{
\|\Lambda^{\frac{\beta}{2}}\theta(\tau)\|_{L^{2}}^{2}\,d\tau}\leq
\|\theta_{0}\|_{L^{2}},\quad\|\theta(t)\|_{L^{p}}\leq \|\theta_{0}\|_{L^{p}},\quad \forall p\in
[2, \infty],
\end{eqnarray}
\begin{eqnarray}\label{t302}
\sup_{0\leq t\leq T}\|u(t)\|_{L^{2}}^{2}+\int_{0}^{T}{
\|\Lambda^{\frac{\alpha}{2}}u(\tau)\|_{L^{2}}^{2}\,d\tau}\leq
C(T,\,u_{0},\,\theta_{0}).
\end{eqnarray}
Now we apply operator $\mbox{curl}$ to the equation $(\ref{Bouss})_{1}$
to obtain the following vorticity $w=\partial_{1}u_{2}-\partial_{2}u_{1}$ equation
\begin{eqnarray}\label{t303}\partial_{t}w+(u\cdot\nabla)w+\Lambda^{\alpha} w=\partial_{x}\theta.\end{eqnarray}
However, the "vortex stretching" term $\partial_{x}\theta$
appears to prevent us from proving any global bound for $w$. To overcome this difficulty, a natural idea is to eliminate the term $\partial_{x}\theta$ from the vorticity equation. This method was first introduced by Hmidi, Keraani and Rousset \cite{HK3,HK4} to treat the Boussinesq equations with critical cases. Now we set $\mathcal {R}_{\alpha}$ as the singular integral operator
$$\mathcal {R}_{\alpha}:=\partial_{x}\Lambda^{-\alpha}.$$
Then we can show that the new quantity $G=\omega-\mathcal {R}_{\alpha}\theta$ satisfies
\begin{eqnarray}\label{t305}\partial_{t}G+(u\cdot\nabla)G+\Lambda^{\alpha}G=[\mathcal {R}_{\alpha},\,u\cdot\nabla]\theta+\Lambda^{\beta-\alpha}\partial_{x}\theta,\end{eqnarray}
here and in sequel,
the following standard commutator notation are used frequently
$$[\mathcal {R}_{\alpha},\,u\cdot\nabla]\theta:=
\mathcal {R}_{\alpha}(u\cdot\nabla\theta)-u\cdot\nabla\mathcal
{R}_{\alpha}\theta.$$
The above equation is very important in our analysis in order to
derive some crucial {\it a priori} estimates. Moreover, the velocity
field $u$ can be decomposed into the following two parts
$$u=\nabla^{\perp}\Delta^{-1}\omega=\nabla^{\perp}\Delta^{-1}G
+\nabla^{\perp}\Delta^{-1}\mathcal {R}_{\alpha}:=u_{G}+u_{\theta}.$$

Before further proving our main result, we need to recall some
useful lemmas. The first lemma concerns the following commutator
estimate, which plays a key role in proving our main result.
\begin{lemma}[see \cite{YX201502}]\label{NCE}
Let $p\in[2, \infty)$ and $r\in[1, \infty]$ and $\delta\in(0,1)$,
$s\in(0, 1)$ such that $s+\delta<1$, then it holds
\begin{align}\label{tNCE}
\|[\Lambda^{\delta},f]g\|_{B_{p,r}^{s}}\leq
C(p,r,\delta,s)\big(\|\nabla f\|_{L^{p}}\|g\|_{
{B}_{\infty,r}^{s+\delta-1}}+\|f\|_{L^{2}}\|g\|_{L^{2}}\big).
\end{align}
Here and in what follows, $B_{p,r}^{s}$ denotes the standard Besov
space.
\end{lemma}

To prove the theorem, we need the following commutator
estimate involving $\mathcal
{R}_{\alpha}$, which was established by Stefanov and Wu \cite{SW}.
\begin{lemma}\label{Lem23}
Assume that $\frac{1}{2}<\alpha<1$ and $1<p_{2}<\infty$,
$1<p_{1},\,p_{3}\leq\infty$ with
$\frac{1}{p_{1}}+\frac{1}{p_{2}}+\frac{1}{p_{3}}=1$. Then for $0\leq
s_{1}<1-\alpha$ and $s_{1}+s_{2}>1-\alpha$, the following holds true
\begin{align}\label{t201}
\Big|\int_{\mathbb{R}^{2}}{ F[\mathcal
{R}_{\alpha},\,u_{G}\cdot\nabla]\theta\,dx}\Big|\leq
C\|\Lambda^{s_{1}}\theta\|_{L^{p_{1}}}\|F\|_{W^{s_{2},\,p_{2}}}
\|G\|_{L^{p_{3}}}.\end{align} Similarly, for $0\leq s_{1}<1-\alpha$
and $s_{1}+s_{2}>2-2\alpha$, the following holds true
\begin{align}\label{t202}
\Big|\int_{\mathbb{R}^{2}}{ F[\mathcal
{R}_{\alpha},\,u_{\theta}\cdot\nabla]H\,dx}\Big|\leq
C\|\Lambda^{s_{1}}\theta\|_{L^{p_{1}}}\|F\|_{W^{s_{2},\,p_{2}}}
\|H\|_{L^{p_{3}}}.\end{align}
Here and in what follows, $W^{s,\,p}$ denotes the standard Sobolev
space.
\end{lemma}
The following lemma is the bilinear estimate which will be used
frequently.
\begin{lemma}\label{Lem25}
Let $2<m<\infty$, $0<s<1$ and $p, q, r\in(1, \infty)^{3}$ such that
$\frac{1}{p}=\frac{1}{q}+\frac{1}{r}$, then it holds
\begin{eqnarray}\label{t205}\|\Lambda^{s}(|f|^{m-2}f)\|_{L^{p}}\leq C\|f\|_{\dot{B}_{q,\,p}^{s}}\|f\|_{L^{r(m-2)}}^{m-2},
\end{eqnarray}
\begin{eqnarray}\||f|^{m-2}f\|_{W^{s,p}}\leq C\|f\|_{{B}_{q,\,p}^{s}}\|f\|_{L^{r(m-2)}}^{m-2}.\nonumber
\end{eqnarray}
\end{lemma}
\begin{proof}[\textbf{Proof of Lemma \ref{Lem25}}]
One can find the proof in \cite{YXX} and we sketch it here for convenience.
Let us recall the following characterization of $\dot{W}^{s,\,p}$ with $0<s<1$
$$\|\Lambda^{s}(|f|^{m-2}f)\|_{L^{p}}^{p}\thickapprox \int_{\mathbb{R}^{2}}{\frac{\||f|^{m-2}f(x+.)-|f|^{m-2}f(.)\|_{L^{p}}^{p}}{|x|^{2+sp}}
\,dx}.$$
Note that the following simple inequality
$$\Big||a|^{m-2}a-|b|^{m-2}b\Big|\leq C(m)|a-b|(|a|^{m-2}+|b|^{m-2}),$$
and H${\rm \ddot{o}}$lder inequality, it results in
\begin{eqnarray*}\||f|^{m-2}f(x+.)-|f|^{m-2}f(.)\|_{L^{p}}&\leq& C\|f(x+.)-f(.)\|_{L^{q}}
\||f|^{m-2}\|_{L^{r}}\nonumber\\
&\leq&C\|f(x+.)-f(.)\|_{L^{q}}
\|f\|_{L^{r(m-2)}}^{m-2}.
\end{eqnarray*}
Thus, it follows from the characterization of Besov space that
\begin{eqnarray*}\|\Lambda^{s}(|f|^{m-2}f)\|_{L^{p}}^{p}&\leq& C \int_{\mathbb{R}^{2}}{\frac{\|f(x+.)-f(.)\|_{L^{q}}^{p}
\|f\|_{L^{r(m-2)}}^{(m-2)p}}{|x|^{2+sp}}
\,dx}\nonumber\\
&\leq&C\|f\|_{L^{r(m-2)}}^{(m-2)p} \int_{\mathbb{R}^{2}}{\frac{\|f(x+.)-f(.)\|_{L^{q}}^{p}
}{|x|^{2+sp}}
\,dx}\nonumber\\
&\leq&C\|f\|_{L^{r(m-2)}}^{(m-2)p} \|f\|_{\dot{B}_{q,\,p}^{s}}^{p}.
\end{eqnarray*}
The H$\rm \ddot{o}$lder inequality directly gives
$$\||f|^{m-2}f\|_{L^{p}}\leq C\|f\|_{L^{q}}\||f|^{m-2}\|_{L^{r}}= C\|f\|_{L^{q}}\|f\|_{L^{r(m-2)}}^{m-2}.$$
Consequently, this concludes the proof of the lemma.
\end{proof}

With the above lemmas in hand, we continue to prove the main result.
First we are now in the position to derive the following estimate
concerning the temperature $\theta$ and $G$, which plays an
important role in proving the main theorem and is also the main
difference compared to the recent manuscript \cite{YXX}.
\begin{lemma}\label{AZL302}
Under the assumptions stated in Theorem \ref{Th1}, let $(u, \theta)$
be the corresponding solution of the system (\ref{Bouss}). If
$\beta>1-\alpha$ and $\alpha>\frac{2}{3}$,
then the temperature $\theta$ admits the following bound for any
$\max\Big\{\frac{2-2\alpha-\beta}{2},\,\,\,\frac{2+\beta-3\alpha}{2}\Big\}
<\delta<\frac{\beta}{2}$
\begin{eqnarray}\label{AZ001}
\sup_{0\leq t\leq T}(\|G(t)\|_{L^{2}}^{2}+
\|\Lambda^{\delta}\theta(t)\|_{L^{2}}^{2})
+\int_{0}^{T}{\big(\|\Lambda^{\frac{\alpha}{2}} G\|_{L^{2}}^{2}
+\|\Lambda^{\delta+\frac{\beta}{2}}\theta\|_{L^{2}}^{2}\big)(\tau)\,d\tau}\leq
C(T,\,u_{0},\,\theta_{0}),\nonumber\\
\end{eqnarray}
where $C(T,\,u_{0},\,\theta_{0})$ is a constant depending on $T$ and
the initial data.
\end{lemma}
\begin{rem}\rm
Although the above estimate (\ref{AZ001}) holds for
$\max\Big\{\frac{2-2\alpha-\beta}{2},\,\,\,\frac{2+\beta-3\alpha}{2}\Big\}
<\delta<\frac{\beta}{2}$, yet by energy estimate (\ref{t301}) and
the classical interpolation, we find that (\ref{AZ001}) is actually
true for any $0\leq\delta<\frac{\beta}{2}$.
\end{rem}

\begin{proof}[\textbf{Proof of Lemma \ref{AZL302}}]
Applying $\Lambda^{\delta}$ ($\delta>0$ to be fixed later) to $(\ref{Bouss})_{2}$, then multiplying it by $\Lambda^{\delta}\theta$, after integration by parts, we find that
\begin{eqnarray}\label{AZ002}\frac{1}{2}\frac{d}{dt}\|\Lambda^{\delta}\theta(t)\|_{L^{2}}^{2}
+\|\Lambda^{\delta+\frac{\beta}{2}}\theta\|_{L^{2}}^{2}=-\int_{\mathbb{R}^{2}}
\Lambda^{\delta}\big(u \cdot
\nabla\theta\big)\Lambda^{\delta}\theta\,dx.
\end{eqnarray}
Hence, an application of the divergence-free condition, commutator
estimate (\ref{tNCE}), Besov embedding and Gagliardo-Nirenberg
inequality directly yields
\begin{eqnarray}&&\Big|\int_{\mathbb{R}^{2}}
\Lambda^{\delta}\big(u \cdot \nabla\theta\big)\Lambda^{\delta}\theta\,dx\Big|
\nonumber\\&=&\Big|\int_{\mathbb{R}^{2}}
[\Lambda^{\delta}, u \cdot \nabla]\theta\,\,\Lambda^{\delta}\theta\,dx\Big|\nonumber\\
&=&\Big|\int_{\mathbb{R}^{2}}
\nabla\cdot[\Lambda^{\delta}, u]\theta\,\,\Lambda^{\delta}\theta\,dx\Big|\nonumber\\
&\leq& C\|\Lambda^{1-\frac{\beta}{2}}[\Lambda^{\delta}, u]\theta\|_{L^{2}}\|\Lambda^{\delta+\frac{\beta}{2}}\theta\|_{L^{2}}\nonumber\\
&\leq& C\|[\Lambda^{\delta}, u ]\theta\|_{H^{1-\frac{\beta}{2}}}\|\Lambda^{\delta+\frac{\beta}{2}}\theta\|_{L^{2}}\nonumber\\
&\leq& C\|[\Lambda^{\delta}, u ]\theta\|_{B_{2,2}^{1-\frac{\beta}{2}}}\|\Lambda^{\delta+\frac{\beta}{2}}\theta\|_{L^{2}}
\nonumber\\
&\leq& C\big(\|\nabla u\|_{L^{2}}\|\theta\|_{B_{\infty,2}^{\delta-\frac{\beta}{2}}}
+\|u\|_{L^{2}}\|\theta\|_{L^{2}}\big)
\|\Lambda^{\delta+\frac{\beta}{2}}\theta\|_{L^{2}}\qquad \Big(\delta<\frac{\beta}{2}\Big)\nonumber\\
&\leq& C\|\omega\|_{L^{2}}\|\theta\|_{L^{\infty}}
\|\Lambda^{\delta+\frac{\beta}{2}}\theta\|_{L^{2}}+C\|u\|_{L^{2}}\|\theta\|_{L^{2}}
\|\Lambda^{\delta+\frac{\beta}{2}}\theta\|_{L^{2}}\nonumber\\
&\leq& C(\|G\|_{L^{2}}+\|\mathcal {R}_{\alpha}\theta\|_{L^{2}})\|\theta\|_{L^{\infty}}
\|\Lambda^{\delta+\frac{\beta}{2}}\theta\|_{L^{2}}+C\|u\|_{L^{2}}\|\theta\|_{L^{2}}
\|\Lambda^{\delta+\frac{\beta}{2}}\theta\|_{L^{2}}\nonumber\\
&\leq& C\|\theta\|_{L^{\infty}}\|G\|_{L^{2}}\|\Lambda^{\delta+\frac{\beta}{2}}\theta\|_{L^{2}}
+C\|\theta\|_{L^{\infty}}\|\Lambda^{1-\alpha}\theta\|_{L^{2}}
\|\Lambda^{\delta+\frac{\beta}{2}}\theta\|_{L^{2}}\nonumber\\&&+C\|u\|_{L^{2}}\|\theta\|_{L^{2}}
\|\Lambda^{\delta+\frac{\beta}{2}}\theta\|_{L^{2}}\nonumber\\&\leq& C\|\theta\|_{L^{\infty}}\|G\|_{L^{2}}\|\Lambda^{\delta+\frac{\beta}{2}}\theta\|_{L^{2}}
+C\|\theta\|_{L^{\infty}}\|\theta\|_{L^{2}}^{\frac{2\delta+\beta+2\alpha-2}{2\delta+\beta}}
\|\Lambda^{\delta+\frac{\beta}{2}}\theta\|_{L^{2}}^{\frac{2-2\alpha}{2\delta+\beta}}
\|\Lambda^{\delta+\frac{\beta}{2}}\theta\|_{L^{2}}\nonumber\\
&&+C\|u\|_{L^{2}}\|\theta\|_{L^{2}}
\|\Lambda^{\delta+\frac{\beta}{2}}\theta\|_{L^{2}}\nonumber\\
&&\Big(\delta>\frac{2-2\alpha-\beta}{2}\Rightarrow \frac{2-2\alpha}{2\delta+\beta}<1\Big)\nonumber\\
&\leq& \frac{1}{2}\|\Lambda^{\delta+\frac{\beta}{2}}\theta\|_{L^{2}}^{2}+
C(\|\theta\|_{L^{\infty}}^{\frac{2(2\delta+\beta)}{2\delta+\beta+2\alpha-2}}
+\|u\|_{L^{2}}^{2})\|\theta\|_{L^{2}}^{2}+C\|\theta\|_{L^{\infty}}^{2}\|G\|_{L^{2}}^{2}.
\nonumber
\end{eqnarray}
Here we have applied the following facts
$$L^{\infty}\hookrightarrow B_{\infty,2}^{\delta-\frac{\beta}{2}},\quad \|\Lambda^{1-\alpha}\theta\|_{L^{2}(\mathbb{R}^{2})}\leq C\|\theta\|_{L^{2}(\mathbb{R}^{2})}^{\frac{2\delta+\beta+2\alpha-2}{2\delta+\beta}}
\|\Lambda^{\delta+\frac{\beta}{2}}\theta\|_{L^{2}(\mathbb{R}^{2})}^{\frac{2-2\alpha}{2\delta+\beta}},$$
which holds true for $\delta<\frac{\beta}{2}$ and $\delta>\frac{2-2\alpha-\beta}{2}$, respectively.\\
Substituting the above estimate into (\ref{AZ002}), we arrive at
\begin{eqnarray}\label{Vor02}\frac{d}{dt}\|\Lambda^{\delta}\theta(t)\|_{L^{2}}^{2}
+\|\Lambda^{\delta+\frac{\beta}{2}}\theta\|_{L^{2}}^{2}\leq
C(\|\theta\|_{L^{\infty}}^{\frac{2(2\delta+\beta)}{2\delta+\beta+2\alpha-2}}
+\|u\|_{L^{2}}^{2})\|\theta\|_{L^{2}}^{2}+C\|\theta\|_{L^{\infty}}^{2}\|G\|_{L^{2}}^{2},
\end{eqnarray}
Now we test the equation (\ref{t305}) by $G$, integrate the
resulting inequality with respect to $x$ and make use of
divergence-free condition to obtain
\begin{eqnarray}\label{AD001}\frac{1}{2}\frac{d}{dt}\|G(t)\|_{L^{2}}^{2}+
\|\Lambda^{\frac{\alpha}{2}}G\|_{L^{2}}^{2}=\int_{\mathbb{R}^{2}}
{[\mathcal {R}_{\alpha},\,u\cdot\nabla]\theta\,\,G\,dx}+\int_{\mathbb{R}^{2}}
{\Lambda^{\beta-\alpha}\partial_{x}\theta\,\,G\,dx}.
\end{eqnarray}
We easily deduce from Gagliardo-Nirenberg inequality and Young
inequality that
\begin{eqnarray}\label{AD002}\Big|\int_{\mathbb{R}^{2}}
{\Lambda^{\beta-\alpha}\partial_{x}\theta\,\,G\,dx}\Big|&\leq&
C\|\Lambda^{\delta+\frac{\beta}{2}}\theta\|_{L^{2}}
\|\Lambda^{1+\frac{\beta}{2}-\alpha-\delta}G\|_{L^{2}}\nonumber\\
&\leq& C\|\Lambda^{\delta+\frac{\beta}{2}}\theta\|_{L^{2}}
\|G\|_{L^{2}}^{\frac{3\alpha+2\delta-2-\beta}{\alpha}}
\|\Lambda^{\frac{\alpha}{2}}G\|_{L^{2}}^{\frac{2+\beta-2\alpha
-2\delta}{\alpha}}\nonumber\\
&&\Big(\frac{2+\beta-3\alpha}{2}<\delta<\frac{2+\beta-2\alpha}{2}\Big)
\nonumber\\
&\leq&\frac{1}{4}\|\Lambda^{\frac{\alpha}{2}}G\|_{L^{2}}^{2}+\frac{1}{4}
\|\Lambda^{\delta+\frac{\beta}{2}}\theta\|_{L^{2}}^{2}+
C\|G\|_{L^{2}}^{2},
\end{eqnarray}
where in the second line, we have used the following
Gagliardo-Nirenberg inequality
$$\|\Lambda^{1+\frac{\beta}{2}-\alpha-\delta}G\|_{L^{2}}\leq C\|G\|_{L^{2}}^{\frac{3\alpha+2\delta-2-\beta}{\alpha}}
\|\Lambda^{\frac{\alpha}{2}}G\|_{L^{2}}^{\frac{2+\beta-2\alpha
-2\delta}{\alpha}},$$ for any
$\frac{2+\beta-3\alpha}{2}<\delta<\frac{2+\beta-2\alpha}{2}$.\\
Observing the decomposition $u=u_{G}+u_{\theta}$, we get
$$\int_{\mathbb{R}^{2}}
{[\mathcal
{R}_{\alpha},\,u\cdot\nabla]\theta\,\,G\,dx}=\int_{\mathbb{R}^{2}}
{[\mathcal
{R}_{\alpha},\,u_{G}\cdot\nabla]\theta\,\,G\,dx}+\int_{\mathbb{R}^{2}}
{[\mathcal {R}_{\alpha},\,u_{\theta}\cdot\nabla]\theta\,\,G\,dx}.$$
Let us use the estimate (\ref{t201}) with $s_{1}=0$ to control the
above first term as
\begin{eqnarray}\label{AD003}\Big|\int_{\mathbb{R}^{2}}
{[\mathcal
{R}_{\alpha},\,u_{G}\cdot\nabla]\theta\,\,G\,dx}\Big|&\leq&
C\|\theta\|_{L^{\infty}}
\|G\|_{L^{2}}\|G\|_{H^{s_{2}}}\quad (s_{2}>1-\alpha)\nonumber\\
&\leq&C\|\theta\|_{L^{\infty}}
\|G\|_{L^{2}}\|G\|_{H^{\frac{\alpha}{2}}}\quad \big(s_{2}\leq\frac{\alpha}{2}\big)\nonumber\\
&\leq&\frac{1}{8}\|\Lambda^{\frac{\alpha}{2}}G\|_{L^{2}}^{2}+
C(1+\|\theta\|_{L^{\infty}}^{2})\|G\|_{L^{2}}^{2}.
\end{eqnarray}
To estimate the second term, we can apply the estimate (\ref{t202})
with $s_{2}=\frac{\alpha}{2}$ to conclude that
\begin{eqnarray}\label{AD004}\Big|\int_{\mathbb{R}^{2}}
{[\mathcal
{R}_{\alpha},\,u_{\theta}\cdot\nabla]\theta\,\,G\,dx}\Big|&\leq&
C\|\theta\|_{L^{\infty}}
\|\theta\|_{{H}^{s_{1}}}\|G\|_{H^{\frac{\alpha}{2}}}\quad \Big(\frac{4-5\alpha}{2} <s_{1}<1-\alpha\Big)\nonumber\\
&\leq&C\|\theta\|_{L^{\infty}}\|\theta\|_{L^{2}}^{\frac{2\delta+\beta-2s_{1}}{2\delta+\beta}}
\|\Lambda^{\delta+\frac{\beta}{2}}\theta\|_{L^{2}}^{\frac{2s_{1}}{2\delta+\beta}}
\|G\|_{H^{\frac{\alpha}{2}}}\nonumber\\&& \Big(0<s_{1}<\delta+\frac{\beta}{2}\Big)\nonumber\\
&\leq&\frac{1}{8}\|\Lambda^{\frac{\alpha}{2}}G\|_{L^{2}}^{2}+\frac{1}{4}
\|\Lambda^{\delta+\frac{\beta}{2}}\theta\|_{L^{2}}^{2}+
C\|\theta\|_{L^{\infty}}^{\frac{4\delta+2\beta}{2\delta+\beta-2s_{1}}}
\|\theta\|_{L^{2}}^{2},
\end{eqnarray}
where in the first line and second line, the number $s_{1}$ should
be satisfied
$$\max\Big\{0,\,\,\frac{4-5\alpha}{2}\Big\}<s_{1}<\min\Big\{1-\alpha,\,\,\delta+
\frac{\beta}{2}\Big\},$$
which can be ensured by choosing $\delta>\frac{4-5\alpha-\beta}{2}$ and $\alpha>\frac{2}{3}$.\\
Inserting above estimates (\ref{AD002})-(\ref{AD004}) into
(\ref{AD001}), we can conclude
\begin{eqnarray}\label{AD005}\frac{1}{2}\frac{d}{dt}\|G(t)\|_{L^{2}}^{2}
+\frac{1}{2} \|\Lambda^{\frac{\alpha}{2}}G\|_{L^{2}}^{2}\leq
\frac{1}{2} \|\Lambda^{\delta+\frac{\beta}{2}}\theta\|_{L^{2}}^{2}
+C(1+\|\theta\|_{L^{\infty}}^{2})\|G\|_{L^{2}}^{2}
+C\|\theta\|_{L^{\infty}}^{\frac{4\delta+2\beta}{2\delta+\beta-2s_{1}}}
\|\theta\|_{L^{2}}^{2}.\nonumber\\
\end{eqnarray}
By putting (\ref{Vor02}) and (\ref{AD005}) together, we finally get
\begin{eqnarray}\label{AD006}&&\frac{d}{dt}(\|G(t)\|_{L^{2}}^{2}+
\|\Lambda^{\delta}\theta(t)\|_{L^{2}}^{2})+\|\Lambda^{\frac{\alpha}{2}}
G\|_{L^{2}}^{2}
+\|\Lambda^{\delta+\frac{\beta}{2}}\theta\|_{L^{2}}^{2}\nonumber\\&\leq&
C(1+\|\theta\|_{L^{\infty}}^{\frac{2(2\delta+\beta)}{2\delta+\beta+2\alpha-2}}
+\|\theta\|_{L^{\infty}}^{\frac{4\delta+2\beta}{2\delta+\beta-2s_{1}}}+\|u\|_{L^{2}}^{2})\|\theta\|_{L^{2}}^{2}+C(1+\|\theta\|_{L^{\infty}}^{2})\|G\|_{L^{2}}^{2},
\end{eqnarray}
for any $\delta$ satisfying
$$\max\Big\{\frac{2-2\alpha-\beta}{2},\,\,\,\frac{2+\beta-3\alpha}{2},\,\,
\frac{4-5\alpha-\beta}{2}\Big\}
<\delta<\min\Big\{\frac{\beta}{2},\,\,\,\frac{\beta+2-2\alpha}{2}\Big\}
=\frac{\beta}{2}.$$
Observing the facts $\alpha>\frac{2}{3}\Rightarrow \frac{4-5\alpha-\beta}{2}<\frac{2-2\alpha-\beta}{2}$ and $\frac{\beta}{2}<\frac{\beta+2-2\alpha}{2}$, the range of $\delta$ becomes
$$\max\Big\{\frac{2-2\alpha-\beta}{2},\,\,\,\frac{2+\beta-3\alpha}{2}\Big\}
<\delta<\frac{\beta}{2}.$$
By the standard Gronwall inequality, we can
easily get from (\ref{AD006}) that
$$\sup_{0\leq t\leq T}(\|G(t)\|_{L^{2}}^{2}+
\|\Lambda^{\delta}\theta(t)\|_{L^{2}}^{2})
+\int_{0}^{T}{\big(\|\Lambda^{\frac{\alpha}{2}} G\|_{L^{2}}^{2}
+\|\Lambda^{\delta+\frac{\beta}{2}}\theta\|_{L^{2}}^{2}\big)(\tau)\,d\tau}\leq
C(T,\,u_{0},\,\theta_{0}).$$ Thus the conclusion is proved.
\end{proof}

Next we establish the following global {\it a priori} bound of
$L^{m}$  norm for $G$ based on Lemma \ref{AZL302}. This {\it a
priori} bound plays a crucial role in proving the main theorem.
\begin{lemma}\label{L302}
Let $\alpha_{0}<\alpha<1$ and $1-\alpha<\beta<
\min\Big\{\frac{\alpha}{2},\,\,
\frac{3\alpha-2}{2\alpha^{2}-6\alpha+5},
\,\,\frac{2-2\alpha}{4\alpha-3}\Big\}.$
Assume that $(u_{0},\,\theta_{0})$ satisfies the assumptions stated in Theorem \ref{Th1}, then the combined equation (\ref{t305}) admits the following bound for any $0\leq t\leq T$
\begin{eqnarray}\label{t331}
\|G(t)\|_{L^{m}}^{m}+
\int_{0}^{T}{\|G(\tau)\|_{L^{\frac{2m}{2-\alpha}}}^{m}\,d\tau}\leq
C(T,\,u_{0},\,\theta_{0}),
\end{eqnarray}
where $m=\frac{2}{2\alpha-1}+\epsilon$ for some $\epsilon>0$ small enough, which may depend on $\alpha$ and $\beta$.
\end{lemma}
\begin{rem}\rm
It follows from the recent paper \cite{YXX} (also \cite{MX}) that we need the key requirement $m>\frac{2}{2\alpha-1}$, but the $m$ can be arbitrarily close to $\frac{2}{2\alpha-1}$. Thus it is sufficient to select $m=\frac{2}{2\alpha-1}+\epsilon$ with any $\epsilon>0$ small enough.
\end{rem}

\begin{proof}[\textbf{Proof of Lemma \ref{L302}}]
To begin with, let us recall the the following fractional version of the Gagliardo-Nirenberg inequality which is due to Hajaiej-Molinet-Ozawa-Wang \cite{HMOW}
\begin{eqnarray}
\|\Lambda^{\gamma\beta}\theta\|_{L^{\frac{1}{\gamma}}}
\leq C\|\Lambda^{\frac{\beta}{2}}\theta\|_{L^{2}}^{2\gamma}
\|\theta\|_{L^{\infty}}^{1-2\gamma},\quad 0<\gamma<\frac{1}{2}.\nonumber
\end{eqnarray}
In fact, the above inequality is a direct consequence of Theorem 1.2 of \cite{HMOW} as well as the equivalence $\dot{W}^{s,\,p}\approx \dot{B}_{p,p}^{s}$ for $0<s\neq \mathbb{N}$ and $1<p<\infty$.\\
Thanks to the bound (\ref{AZ001}), we have for any $0<\gamma<\frac{1}{2}$ and $2\leq q<\infty$
\begin{eqnarray}\label{t314}
\int_{0}^{T}{\|\Lambda^{\gamma\beta}\theta(t)\|_{L^{\frac{1}{\gamma}}}^{q}\,dt}
&\leq& C\|\theta_{0}\|_{L^{\infty}}^{(1-2\gamma)q}\int_{0}^{T}{\|\Lambda^{\frac{\beta}{2}}\theta(t)\|_{L^{2}}^{2\gamma q}\,dt}\nonumber\\
&\leq& C\|\theta_{0}\|_{L^{\infty}}^{(1-2\gamma)q}\int_{0}^{T}{
\|\Lambda^{\delta}\theta(t)\|_{L^{2}}^{\frac{4\delta\gamma
q}{\beta}}
\|\Lambda^{\delta+\frac{\beta}{2}}\theta(t)\|_{L^{2}}^{\frac{2(\beta-2\delta)\gamma q}{\beta}}\,dt}\nonumber\\
&\leq& C\|\theta_{0}\|_{L^{\infty}}^{(1-2\gamma)q}\sup_{0\leq t\leq
T}\|\Lambda^{\delta}\theta(t)\|_{L^{2}}^{\frac{4\delta\gamma
q}{\beta}}\int_{0}^{T}{
\|\Lambda^{\delta+\frac{\beta}{2}}\theta(t)\|_{L^{2}}^{\frac{2(\beta-2\delta)\gamma q}{\beta}}\,dt}\nonumber\\
&\leq& C(T,\,u_{0},\,\theta_{0}),
\end{eqnarray}
where in the last line we just take $\delta$ such that $\min\big\{\frac{\beta}{2}(1-\frac{1}{q\gamma}),\,\,0\big\}\leq \delta<\frac{\beta}{2}$.\\
Multiplying the equation (\ref{t305}) by
$|G|^{m-2}G$ ($m=\frac{2}{2\alpha-1}+\epsilon$ and $\epsilon>0$ to be fixed later),
we have after integration by part and using the divergence-free condition
\begin{eqnarray}\label{t315}&&\frac{1}{m}\frac{d}{dt}\|G(t)\|_{L^{m}}^{m}
+\int_{\mathbb{R}^{2}}
(\Lambda^{\alpha}G)|G|^{m-2}G\,dx\nonumber\\&=&\int_{\mathbb{R}^{2}}
{[\mathcal
{R}_{\alpha},\,u\cdot\nabla]\theta\,\,|G|^{m-2}G\,dx}+\int_{\mathbb{R}^{2}}
{\Lambda^{\beta-\alpha}\partial_{x}\theta\,\,|G|^{m-2}G\,dx}\nonumber\\&=&\int_{\mathbb{R}^{2}}
{[\mathcal
{R}_{\alpha},\,u_{G}\cdot\nabla]\theta\,\,|G|^{m-2}G\,dx}+\int_{\mathbb{R}^{2}}
{[\mathcal
{R}_{\alpha},\,u_{\theta}\cdot\nabla]\theta\,\,|G|^{m-2}G\,dx}\nonumber\\&&
+\int_{\mathbb{R}^{2}}
{\Lambda^{\beta-\alpha}\partial_{x}\theta\,\,|G|^{m-2}G\,dx}.
\end{eqnarray}
We infer from the maximum principle and Sobolev embedding that
\begin{eqnarray}\label{t316}\int_{\mathbb{R}^{2}}
(\Lambda^{\alpha}G)|G|^{m-2}G\,dx\geq \widetilde{C}\|\Lambda^{\frac{\alpha}{2}}G^{\frac{m}{2}}\|_{L^{2}}^{2}\geq \widetilde{C}\|G\|_{L^{\frac{2m}{2-\alpha}}}^{m},\end{eqnarray}
where $\widetilde{C}>0$ is an absolute constant.\\
Taking into account the inequality (\ref{t205}), we find that
\begin{eqnarray}\label{t317}\Big|\int_{\mathbb{R}^{2}}
{\Lambda^{\beta-\alpha}\partial_{x}\theta\,\,|G|^{m-2}G\,dx}\Big|&\leq&
C\|\Lambda^{\gamma\beta}\theta\|_{L^{\frac{1}{\gamma}}}
\|\Lambda^{1-\alpha+(1-\gamma)\beta}(|G|^{m-2}G)
\|_{L^{\frac{1}{1-\gamma}}}\nonumber\\
&\leq& C\|\Lambda^{\gamma\beta}\theta\|_{L^{\frac{1}{\gamma}}}
\|G\|_{B_{2,\,\frac{1}{1-\gamma}}^{1-\alpha+(1-\gamma)\beta}}\|G\|_{L^{\frac{2(m-2)}{1-2
\gamma}}}^{m-2}\nonumber\\
&\leq&C\|\Lambda^{\gamma\beta}\theta\|_{L^{\frac{1}{\gamma}}} \|G
\|_{H^{\frac{\alpha}{2}}}\|G\|_{L^{\frac{2(m-2)}{1-2
\gamma}}}^{m-2},
\end{eqnarray}
where we have used $H^{\frac{\alpha}{2}}\hookrightarrow B_{2,\,\frac{1}{1-\gamma}}^{1-\alpha+(1-\gamma)\beta}$ and $1-\alpha+(1-\gamma)\beta<\frac{\alpha}{2}$,
namely \begin{eqnarray}\label{Con1}\gamma>\frac{2\beta+2-3\alpha}{2\beta}.\end{eqnarray}
Now the estimate (\ref{t202}) with
$s_{1}=\gamma\beta$ implies that
\begin{eqnarray}\label{t318}&&\Big|\int_{\mathbb{R}^{2}}
{[\mathcal
{R}_{\alpha},\,u_{\theta}\cdot\nabla]\theta\,\,|G|^{m-2}G\,dx}\Big|\nonumber\\&\leq&
C\|\Lambda^{\gamma\beta}\theta\|_{L^{\frac{1}{\gamma}}}\|\theta\|_{L^{\infty}}
\||G|^{m-2}G\|_{W^{s_{2},\,\frac{1}{1-\gamma}}}\nonumber\\
&&\big(s_{2}>2-2\alpha-\gamma\beta,\,\,\,0\leq\gamma\beta<1-\alpha\big)\nonumber\\
&\leq&C\|\Lambda^{\gamma\beta}\theta\|_{L^{\frac{1}{\gamma}}}
\|\theta_{0}\|_{L^{\infty}} \|G\|_{B_{2,\,\frac{1}{1-\gamma}}^{s_{2}}}\|G\|_{L^{\frac{2(m-2)}{1-2
\gamma}}}^{m-2}\nonumber\\
&\leq&C\|\Lambda^{\gamma\beta}\theta\|_{L^{\frac{1}{\gamma}}}
\|G\|_{H^{\frac{\alpha}{2}}}\|G\|_{L^{\frac{2(m-2)}{1-2
\gamma}}}^{m-2},\quad \big(s_{2}<\frac{\alpha}{2}\big).
\end{eqnarray}
Now we verify that the number of above $s_{2}$ can be achieved. Indeed,
it sufficient to select $\gamma$ as follows
\begin{eqnarray}\label{Con2} 0\leq\gamma\beta<1-\alpha,\quad 2-2\alpha-\gamma\beta<\frac{\alpha}{2}.\end{eqnarray}
According to inequality (\ref{t201}) with $s_{1}=0$ as well as inequality (\ref{t205}), it gives
\begin{eqnarray}\label{t319}&&\Big|\int_{\mathbb{R}^{2}}
{[\mathcal
{R}_{\alpha},\,u_{G}\cdot\nabla]\theta\,\,|G|^{m-2}G\,dx}\Big|\nonumber\\&\leq&
C\|G\|_{L^{q}}\|\theta\|_{L^{\infty}}
\||G|^{m-2}G\|_{W^{s_{2}-\frac{\widetilde{\delta}}{2},\,p}}\quad \big(s_{2}-\frac{\widetilde{\delta}}{2}>1-\alpha,\,\,\frac{1}{p}+\frac{1}{q}=1\big)\nonumber\\
&&(  \mbox{ $\widetilde{\delta}>0$ is small enough})\nonumber\\
&\leq&C\|\theta_{0}\|_{L^{\infty}}
\|G\|_{L^{q}}\|G\|_{L^{(m-2)\times\frac{q}{m-2}}}^{m-2}
\|G\|_{B_{\frac{q}{q-(m-1)},\,p}^{s_{2}-\frac{\widetilde{\delta}}{2}}}
\quad(q>m-1)\nonumber\\
&\leq&C\|\theta_{0}\|_{L^{\infty}}
\|G\|_{L^{q}}^{m-1}\|G\|_{B_{\frac{q}{q-(m-1)},\,p}^{s_{2}-\frac{\widetilde{\delta}}{2}}}\nonumber\\
&\leq&C\|\theta_{0}\|_{L^{\infty}}
\|G\|_{L^{q}}^{m-1}\|G\|_{H^{s_{2}-1+\frac{2(m-1)}{q}}}\quad \big(q\leq 2(m-1)\big),
\end{eqnarray}
where we have applied $H^{s_{2}-1+\frac{2(m-1)}{q}}\hookrightarrow B_{\frac{q}{q-(m-1)},\,p}^{s_{2}-\frac{\widetilde{\delta}}{2}}$ for $m-1<q\leq 2(m-1)$.
Thanks to the requirement $s_{2}-\frac{\widetilde{\delta}}{2}>1-\alpha$ in (\ref{t319}), we can choose a sufficiently small $\widetilde{\delta}>0$ (in fact we can take $\widetilde{\delta}\leq\frac{4\alpha-3}{8}$ for example to satisfy all the conditions) such that
$$s_{2}=1-\alpha+\widetilde{\delta}.$$
Notice that the following interpolation inequality
$$\|G\|_{H^{-\alpha+\widetilde{\delta}+\frac{2(m-1)}{q}}}\leq C \|G\|_{L^{2}}^{1-\mu}
\|G\|_{H^{\frac{\alpha}{2}}}^{\mu},$$
where $$ \mu=\frac{-2\alpha+2\widetilde{\delta}+\frac{4(m-1)}{q}}{\alpha},\quad
 \frac{4(m-1)}{3\alpha-2\widetilde{\delta}}\leq q\leq \frac{2(m-1)}{\alpha-\widetilde{\delta}},$$
one can conclude that
\begin{eqnarray}\label{t320}\Big|\int_{\mathbb{R}^{2}}
{[\mathcal
{R}_{\alpha},\,u_{G}\cdot\nabla]\theta\,\,|G|^{m-2}G\,dx}\Big|
&\leq & C\|\theta_{0}\|_{L^{\infty}}\|G\|_{L^{q}}^{m-1}\|G\|_{L^{2}}^{1-\mu}
\|G\|_{H^{\frac{\alpha}{2}}}^{\mu}\nonumber
\\&\leq&C\|G\|_{L^{q}}^{m-1}
\|G\|_{H^{\frac{\alpha}{2}}}^{\mu}.
\end{eqnarray}
Substituting the estimates (\ref{t316})-(\ref{t318}) and
(\ref{t320}) into (\ref{t315}), one arrives at
\begin{align}\label{t321}\frac{d}{dt}\|G(t)\|_{L^{m}}^{m}+
\|G\|_{L^{\frac{2m}{2-\alpha}}}^{m}\leq C\|\Lambda^{\gamma\beta}\theta\|_{L^{\frac{1}{\gamma}}}
\|G\|_{H^{\frac{\alpha}{2}}}\|G\|_{L^{\frac{2(m-2)}{1-2
\gamma}}}^{m-2}+C\|G\|_{L^{q}}^{m-1}
\|G\|_{H^{\frac{\alpha}{2}}}^{\mu}.
\end{align}
By the Gagliardo-Nirenberg inequalities, we know
\begin{eqnarray}\label{t322}
\|G\|_{L^{\frac{2(m-2)}{1-2
\gamma}}}\leq C\|G\|_{L^{m}}^{1-\lambda_{1}}
\|G\|_{L^{\frac{2m}{2-\alpha}}}^{\lambda_{1}},\quad \lambda_{1}=\frac{(1+2\gamma)m-4}{\alpha(m-2)},
\end{eqnarray}
\begin{eqnarray}\label{t323}
\|G\|_{L^{q}}\leq C\|G\|_{L^{m}}^{1-\lambda_{2}}
\|G\|_{L^{\frac{2m}{2-\alpha}}}^{\lambda_{2}},\quad \lambda_{2}=\frac{2-\frac{2m}{q}}{\alpha}.
\end{eqnarray}
Here we want to emphasize that the following restrictions
\begin{eqnarray}\label{Con3}\frac{4-m}{2m}\leq\gamma\leq
\frac{m-(2-\alpha)(m-2)}{2m}, \quad m\leq q\leq
\frac{2m}{2-\alpha}\end{eqnarray}implies $0\leq \lambda_{1}\leq 1$
and $0\leq \lambda_{2}\leq 1$, respectively.\\
In view of above interpolation inequalities (\ref{t322}) and
(\ref{t323}), we can obtain
\begin{eqnarray}\label{t324}
&&C\|\Lambda^{\gamma\beta}\theta\|_{L^{\frac{1}{\gamma}}} \|G
\|_{H^{\frac{\alpha}{2}}}\|G\|_{L^{\frac{2(m-2)}{1-2
\gamma}}}^{m-2}\nonumber\\&\leq&
C\|\Lambda^{\gamma\beta}\theta\|_{L^{\frac{1}{\gamma}}} \|G
\|_{H^{\frac{\alpha}{2}}}\|G\|_{L^{m}}^{(m-2)(1-\lambda_{1})}
\|G\|_{L^{\frac{2m}{2-\alpha}}}^{(m-2)\lambda_{1}}\nonumber\\
&\leq& \frac{1}{4}\|G\|_{L^{\frac{2m}{2-\alpha}}}^{m}+
C(\|\Lambda^{\gamma\beta}\theta\|_{L^{\frac{1}{\gamma}}} \|G\|_{H^{\frac{\alpha}{2}}})^{\frac{m}{m-(m-2)\lambda_{1}}}
\|G\|_{L^{m}}^{\frac{m(m-2)(1-\lambda_{1})}{m-(m-2)\lambda_{1}}},
\end{eqnarray}
\begin{eqnarray}\label{t325}
&&C\|G\|_{L^{q}}^{m-1}
\|G\|_{H^{\frac{\alpha}{2}}}^{\mu}\nonumber
\\&\leq&
C\|G\|_{L^{m}}^{(m-1)(1-\lambda_{2})}
\|G\|_{L^{\frac{2m}{2-\alpha}}}^{(m-1)\lambda_{2}}
\|G\|_{H^{\frac{\alpha}{2}}}^{\mu}\nonumber\\
&\leq& \frac{1}{4}\|G\|_{L^{\frac{2m}{2-\alpha}}}^{m}+
C\|G\|_{H^{\frac{\alpha}{2}}}^{\frac{m\mu}{m-(m-1)\lambda_{2}}}
\|G\|_{L^{m}}^{\frac{m(m-1)(1-\lambda_{2})}{m-(m-1)\lambda_{2}}}.
\end{eqnarray}
Inserting the estimates (\ref{t324}) and (\ref{t325}) into (\ref{t321}), it holds that
\begin{eqnarray}\frac{d}{dt}\|G(t)\|_{L^{m}}^{m}+
\|G\|_{L^{\frac{2m}{2-\alpha}}}^{m}&\leq& C(\|\Lambda^{\gamma\beta}\theta\|_{L^{\frac{1}{\gamma}}} \|G\|_{H^{\frac{\alpha}{2}}})^{\frac{m}{m-(m-2)\lambda_{1}}}
\|G\|_{L^{m}}^{\frac{m(m-2)(1-\lambda_{1})}{m-(m-2)\lambda_{1}}}
\nonumber\\&&+C\|G\|_{H^{\frac{\alpha}{2}}}^{\frac{m\mu}{m-(m-1)\lambda_{2}}}
\|G\|_{L^{m}}^{\frac{m(m-1)(1-\lambda_{2})}{m-(m-1)\lambda_{2}}}.\nonumber
\end{eqnarray}
By direct calculation, we have the following facts
$$\frac{m(m-2)(1-\lambda_{1})}{m-(m-2)\lambda_{1}}\leq m,\quad\frac{m(m-1)(1-\lambda_{2})}{m-(m-1)\lambda_{2}}\leq m,$$
$$m\leq\frac{2}{2-2\alpha+\widetilde{\delta}}\Rightarrow \frac{m\mu}{m-(m-1)\lambda_{2}}\leq 2,$$
and
\begin{eqnarray}\label{Con4}\gamma< \frac{8-(2-\alpha)m}{4m}\,\,\,\Big(m<\frac{8}{2-\alpha}\Big)\Rightarrow \frac{m}{m-(m-2)\lambda_{1}}< 2.\end{eqnarray}
We thus get
\begin{eqnarray}\label{t326}\frac{d}{dt}\|G(t)\|_{L^{m}}^{m}\leq C\Big\{\big(\|\Lambda^{\gamma\beta}\theta\|_{L^{\frac{1}{\gamma}}} \|G\|_{H^{\frac{\alpha}{2}}}\big)^{\frac{m}{m-(m-2)\lambda_{1}}}
+\|G\|_{H^{\frac{\alpha}{2}}}^{\frac{m\mu}{m-(m-1)\lambda_{2}}}\Big\}(1+
\|G\|_{L^{m}}^{m}).
\end{eqnarray}
Thanks to the above facts (\ref{Con4}) as well as the bound (\ref{t314}), we can deduce that
$$(\|\Lambda^{\gamma\beta}\theta\|_{L^{\frac{1}{\gamma}}} \|G\|_{H^{\frac{\alpha}{2}}})^{\frac{m}{m-(m-2)\lambda_{1}}}\in L^{1}(0,\,T),\qquad\|G\|_{H^{\frac{\alpha}{2}}}^{\frac{m\mu}{m-(m-1)\lambda_{2}}}\in L^{1}(0,\,T).$$
By the Gronwall inequality, we can deduce from (\ref{t326}) that
\begin{eqnarray}\label{New01}\|G(t)\|_{L^{m}}^{m}+
\int_{0}^{T}{\|G(\tau)\|_{L^{\frac{2m}{2-\alpha}}}^{m}\,d\tau}\leq
C<\infty.\end{eqnarray}
Finally, let us check that all the restrictions would work.
Combining all the requirement on the number $q$, it should be
$$\max\Big\{m-1,\,\,\frac{4(m-1)}{3\alpha-2\widetilde{\delta}},\,\,m\Big\}<q< \min\Big\{2(m-1),\,\,\frac{2(m-1)}{\alpha-\widetilde{\delta}},\,\,\frac{2m}{2-\alpha}\Big\}.$$
Direct computations yields that the number $q$ can be fixed if we select $\widetilde{\delta}<\frac{3\alpha-2}{2}$.\\
Putting all the restrictions (\ref{Con1}), (\ref{Con2}), (\ref{Con3}), (\ref{Con4}) and $0<\gamma<\frac{1}{2}$ on $\gamma$, we have
\begin{eqnarray}\label{Con5}\mathcal{\underline{B}}(\alpha)<\gamma<
\mathcal{\overline{B}}(\alpha)
,\end{eqnarray}
where $$\mathcal{\underline{B}}(\alpha)=\max\Big\{0,\,\,\frac{2\beta+2-3\alpha}{2\beta},
\,\,\frac{4-5\alpha}{2\beta},\,\,\frac{4-m}{2m}\Big\},$$
$$\mathcal{\overline{B}}(\alpha)=\min\Big\{\frac{1}{2},\,\,\frac{1-\alpha}
{\beta},\,\,\frac{m-(2-\alpha)(m-2)}{2m},\,\,
\frac{8-(2-\alpha)m}{4m}\Big\},$$
and
$$2<m<\min\Big\{4,\,\,\frac{2}{2-2\alpha+\widetilde{\delta}},\,\,\frac{8}{2-\alpha}\Big\}=4.$$
According to $\beta>1-\alpha$ and $m<4$, the $\mathcal{\underline{B}}(\alpha)$ and $\mathcal{\overline{B}}(\alpha)$ can be reduced to
$$\mathcal{\underline{B}}(\alpha)=\max\Big\{0,\,\,\frac{2\beta+2-3\alpha}{2\beta},
\,\,\frac{4-m}{2m}\Big\},$$
$$\mathcal{\overline{B}}(\alpha)=\min\Big\{\frac{1}{2},\,\,\frac{1-\alpha}
{\beta},\,\,\frac{m-(2-\alpha)(m-2)}{2m}\Big\}.$$
Therefore, the $\gamma$ would work if the restriction on $\beta$ satisfies
\begin{eqnarray}\label{New05}1-\alpha<\beta<\min\Big\{\frac{\alpha}{2},\,\,\frac{(3\alpha-2)m}{m+(2-\alpha)(m-2)}
,\,\,\frac{2(1-\alpha)m}{4-m}\Big\}.
\end{eqnarray}
Notice that the above inequality (\ref{New05}) is strict inequality and the following key requirement
\begin{eqnarray}\label{A010}m>\frac{2}{2\alpha-1},\end{eqnarray}
we just verify that the above inequality (\ref{New05}) holds true when $m=\frac{2}{2\alpha-1}$. In this case, substituting the number $m=\frac{2}{2\alpha-1}$ into (\ref{New05}), the inequality (\ref{New05}) reduces to
\begin{eqnarray}\label{NR2}1-\alpha
<\beta<\min\Big\{\frac{\alpha}{2},\,\,
\frac{3\alpha-2}{2\alpha^{2}-6\alpha+5},
\,\,\frac{2-2\alpha}{4\alpha-3}\Big\}.
\end{eqnarray}
By tedious computations, it is not difficult to check that $\beta$ would work as long as $$1-\alpha<\frac{3\alpha-2}{2\alpha^{2}-6\alpha+5}\Rightarrow\alpha>\alpha_{0}.$$
If the above inequality (\ref{New05}) holds true when $m=\frac{2}{2\alpha-1}$, then one may take $m=\frac{2}{2\alpha-1}+\epsilon$ for some sufficiently small $\epsilon$ ($\epsilon>0$ may depend on $\alpha$ and $\beta$) such that both inequalities (\ref{New05}) and (\ref{A010}) fulfil. The reason is that both the inequalities (\ref{New05}) and (\ref{A010}) are strict.
\end{proof}

Now let us say some words to complete the proof of Theorem
\ref{Th1}.
\begin{proof}[\textbf{Proof of Theorem \ref{Th1}}]
In Lemma \ref{L302}, we have proved that
\begin{eqnarray}\label{242}\sup_{0\leq t\leq T}\|G(t)\|_{L^{\frac{2}{2\alpha-1}
+\epsilon}}<\infty,\end{eqnarray} which is a key estimate in order
to complete the proof of Theorem \ref{Th1} (see for example \cite{MX,YXX}). For the sake of convenience, we sketch it here. In fact, as detailed in Step 2 of
\cite{YXX}, the above estimate (\ref{242}) implies
$$\int_{0}^{T}{\|\omega(\tau)\|_{L^{\frac{2}{2\alpha-1}
+\epsilon}}\,d\tau}<\infty,$$ which further gives rise to
$$\int_{0}^{T}{\|G(\tau)\|_{B_{\infty,1}^{0}}\,d\tau}<\infty.$$
Finally, by Lemma 3.3 of \cite{YXX}, we obtain
$$\int_{0}^{T}{\|\omega(\tau)\|_{B_{\infty,1}^{0}}\,d\tau}<\infty.$$
It follows from the Littlewood-Paley technique that
\begin{eqnarray}
\int_{0}^{T}{\|\nabla u(\tau)\|_{L^{\infty}}\,d\tau}\leq C
\int_{0}^{T}{(\|u(\tau)\|_{L^{2}}+\|\omega(\tau)\|_{B_{\infty,1}^{0}})\,d\tau}<\infty.\nonumber
\end{eqnarray}
The above estimate is sufficient for us to get the desired results
of Theorem \ref{Th1}. The details can be found in \cite{MX,YXX}. Thus
we omit the details. Therefore, this concludes the proof of Theorem
\ref{Th1}.
\end{proof}


\vskip .4in
\begin{thebibliography}{00} \frenchspacing
\bibitem{ACSWXY}
D. Adhikari, C. Cao, H. Shang, J. Wu, X. Xu, Z. Ye, Global regularity results for the 2D Boussinesq equations with partial dissipation, submitted for publication (2014).
\bibitem{ACW10}
D. Adhikari, C.  Cao, J. Wu, The 2D Boussinesq equations with vertical viscosity and vertical diffusivity, J. Differential Equations  249 (2010) 1078-1088.
\bibitem{ACW11}
D. Adhikari, C.  Cao, J. Wu,  Global regularity results for the 2D Boussinesq equations with vertical dissipation, J. Differential Equations  251 (2011) 1637-1655.
\bibitem{ACWX}
D. Adhikari, C. Cao, J. Wu, X. Xu, Small global solutions to the damped two-dimensional Boussinesq equations, J. Differential Equations  256 (2014)  3594-3613.
\bibitem{Can}
J. Cannon, E. DiBenedetto, The initial value problem for the Boussinesq equation
with data in $L^{p}$, Lecture Notes in Mathematics, Vol. 771. Springer, Berlin, (1980), 129-144.
\bibitem{CW}
C. Cao, J. Wu, Global regularity for the 2D anisotropic Boussinesq
equations with vertical dissipation, Arch. Rational Mech 208 (2013),
985-1004.
\bibitem{Cap}
M. Caputo, Linear models of dissipation whose Q is almost frequency independent-II, Geophy. J. R. Astr. Soc. 13 (1967), 529-539.
\bibitem{C1}
D. Chae, Global regularity for the 2D Boussinesq equations with
partial viscosity terms, Adv. Math. 203 (2006), 497-513.
\bibitem{CW2}
D. Chae, J. Wu, The 2D Boussinesq equations with logarithmically
supercritical velocities, Adv. Math. 230 (2012), 1618-1645.
\bibitem{CLA}
P. Clavin, Instabilities and nonlinear patterns of overdriven detonations in gases, in: H. Berestycki, Y. Pomeau (Eds.), Nonlinear PDEs in Condensed Matter and Reactive
Flows, Kluwer, (2002), 49-97.
\bibitem{Constantin}
P. Constantin, A. Majda, and E. Tabak, Formation of strong fronts in the 2-D quasi-geostrophic thermal active scalar, Nonlinearity 7 (1994), 1495-1533.
\bibitem{CV}
P. Constantin, V. Vicol, Nonlinear maximum principles for dissipative linear nonlocal
operators and applications, Geom. Funct. Anal. 22 (2012), 1289-1321.
\bibitem{CDJ}
X. Cui, C. Dou, Q. Jiu, Local well-posedness and blow up criterion for the inviscid Boussinesq system in H$\rm\ddot{o}$lder spaces, J. Partial Differential Equations 25 (2012), 220-238.
\bibitem{D}
R. Danchin, Remarks on the lifespan of the solutions to some models
of incompressible fluid mechanics, Proc. Amer. Math. Soc. 141
(2013), 1979-1993.
\bibitem{DP3}
R. Danchin, M. Paicu, Global existence results for the anisotropic
Boussinesq system in dimension two, Math. Models Methods Appl. Sci.
21 (2011), 421-457.
\bibitem{DI}
J. Droniou, C. Imbert, Fractal first-order partial differential equations, Arch. Ration. Mech. Anal. 182 (2006), 299-331.
\bibitem{Fan}
S. Fan, A new extracting formula and a new distinguishing means on the one variable cubic equation, J. Hainan Teach. Coll 2 (1989) 91-98.
\bibitem{HMOW}
H. Hajaiej, L. Molinet, T. Ozawa, B. Wang, Sufficient and necessary conditions for the fractional Gagliardo-Nirenberg inequalities and applications to
Navier-Stokes and generalized Boson equations, in: T. Ozawa, M. Sugimoto (Eds.), RIMS Kkyroku Bessatsu B26: Harmonic Analysis and Nonlinear
Partial Differential Equations, Vol. 5, 2011, pp. 159-175.
\bibitem{HassHm}
Z. Hassainia, T. Hmidi, On the inviscid Boussinesq system with rough initial data, J. Math. Anal. Appl. 430 (2015), 777-809.
\bibitem{HK3}
T. Hmidi, S. Keraani, F. Rousset, Global well-posedness for a
Boussinesq-Navier-Stokes system with critical dissipation, J.
Differential Equations 249 (2010), 2147-2174.
\bibitem{HK4}
T. Hmidi, S. Keraani, F. Rousset, Global well-posedness for
Euler-Boussinesq system with critical dissipation, Comm. Partial
Differential Equations 36 (2011), 420-445.
\bibitem{Hmidi2011}
T. Hmidi, On a maximum principle and its application to the logarithmically critical Boussinesq system, Anal. PDE (2011), 247-284.
\bibitem{HL}
T. Y. Hou, C. Li, Global well-posedness of the viscous Boussinesq
equations, Discrete Contin. Dyn. Syst. 12 (2005),
1-12.
\bibitem{JPL}
J. Jia, J. Peng, K. Li, On the global well-posedness of a generalized 2D Boussinesq equations, NoDEA Nonlinear Differential
Equations Appl.  (in press).
\bibitem{JMWZ}
Q. Jiu, C. Miao, J. Wu, Z. Zhang, The 2D incompressible Boussinesq
equations with general critical dissipation,  SIAM J. Math. Anal. 46 (2014) 3426-3454.
\bibitem{JWYang}
Q. Jiu, J. Wu, W. Yang, Eventual regularity of the two-dimensional
Boussinesq equations with supercritical dissipation, J. Nonlinear Science 25 (2015), 37-58.
\bibitem{KRTW}
D. KC, D. Regmi, L. Tao, J. Wu, The 2D Euler-Boussinesq equations with a singular velocity, J. Differential Equations  257 (2014) 82-108.
\bibitem{LLT}
A. Larios, E. Lunasin, E.S. Titi, Global well-posedness for the 2D Boussinesq system with anisotropic viscosity and without heat diffusion, J. Differential
Equations 255 (2013), 2636-2654.
\bibitem{LT}
J. Li, E.S. Titi, Global well-posedness of the 2D Boussinesq equations with vertical dissipation, arXiv:1502.06180.
\bibitem{LWZ}
X. Liu, M. Wang, Z. Zhang, Local well-posedness and blowup criterion of the Boussinesq equations in critical Besov spaces, J. Math. Fluid Mech. 12 (2010) 280-292.
\bibitem{MB}
 A. Majda, A. Bertozzi, Vorticity and Incompressible
Flow, Cambridge University Press, Cambridge, 2001.
\bibitem{MX}
C. Miao, L. Xue, On the global well-posedness of a class of
Boussinesq-Navier-Stokes systems, NoDEA Nonlinear Differential
Equations Appl. 18 (2011), 707-735.
\bibitem{PG}
J. Pedlosky, Geophysical fluid dynamics, New York, Springer-Verlag,
1987.
\bibitem{SW}
A. Stefanov, J. Wu, A global regularity result for the 2D Boussinesq
equations with critical dissipation, arXiv:1411.1362v3 [math.AP].
\bibitem{WuXu}
J. Wu, X. Xu, Well-posedness and inviscid limits of the Boussinesq equations with fractional Laplacian dissipation,  Nonlinearity  (2014) 2215-2232.
\bibitem{WXY}
J. Wu, X. Xu, Z. Ye, Global smooth solutions to the n-dimensional damped models of incompressible fluid mechanics with small initial datum, J. Nonlinear Science 25 (2015), 157-192.
\bibitem{XX}
X. Xu, Global regularity of solutions of 2D Boussinesq equations
with fractional diffusion,  Nonlinear Anal. 72 (2010),
677-681.
\bibitem{YJW}
W. Yang, Q. Jiu, J. Wu, Global well-posedness for a class of 2D Boussinesq systems with fractional dissipation, J. Differential Equations, 257 (2014) 4188-4213.
\bibitem{YZ}
Z. Ye, Blow-up criterion of smooth solutions for the
Boussinesq equations, Nonlinear Anal. 110 (2014), 97-103.
\bibitem{YX201501}
Z. Ye, X. Xu, Remarks on global regularity of the 2D Boussinesq
equations with fractional dissipation, Nonlinear Anal. (in press).
\bibitem{YX201502}
Z. Ye, X. Xu, Global well-posedness of the 2D Boussinesq equations
with fractional Laplacian dissipation, arXiv:1506.00470v1 [math.AP].
\bibitem{YXX}
Z. Ye, X. Xu, L. Xue, On the global regularity of the 2D Boussinesq equations with fractional dissipation, submitted for publication (2014).
\end{thebibliography}
\end{document}